\documentclass[11pt,reqno]{amsart}

\usepackage[utf8]{inputenc}
\usepackage[T1]{fontenc}
\usepackage[french,english]{babel}
\usepackage{lmodern,mathtools,hyperref}
\usepackage[a4paper,margin=2.5cm]{geometry}

\numberwithin{equation}{section}

\author{Michel Bonnefont}
\address[MB]{%
  Institut de Mathématiques de Bordeaux, Université de Bordeaux, France.}

\author{Djalil Chafaï}
\address[DC]{CEREMADE, Université Paris-Dauphine, PSL, IUF, France.}

\author{Ronan Herry}
\address[RH]{Université du Luxembourg et Université Paris-Est Marne-la-Vallée.}

\title[On logarithmic Sobolev inequalities on the Heisenberg group]%
{On logarithmic Sobolev inequalities\\%
for the heat kernel on the Heisenberg group}

\date{July 2016, revised September 2017, revised March 2018, to appear in \emph{\href{Annales de la Faculté des sciences de Toulouse : Mathématiques, Série 6, Tome 29 (2020) no. 2, pp. 335-355. doi : 10.5802/afst.1633. https://afst.centre-mersenne.org/item/AFST_2020_6_29_2_335_0/}{Annales de la Faculté des sciences de Toulouse : Mathématiques, Série 6, Tome 29 (2020) no. 2, pp. 335-355}}, compiled \today}


\newtheorem{theorem}{Theorem}[section]
\newtheorem{lemma}[theorem]{Lemma}
\newtheorem{corollary}[theorem]{Corollary}

\DeclareMathOperator{\Dil}{dil}

\DeclareMathOperator{\Var}{\mathbf{Var}}
\DeclareMathOperator{\Ent}{\mathbf{Ent}}
\DeclareMathOperator{\Esp}{\mathbf{E}}

\DeclareMathOperator{\diag}{diag}

\DeclareMathOperator{\e}{e}

\newcommand{\dd}{\,\mathrm{d}}
\newcommand{\veps}{\varepsilon}

\newcommand{\cC}{\mathcal{C}}

\newcommand{\cN}{\mathcal{N}}

\newcommand{\bA}{\mathbf{A}}\newcommand{\bB}{\mathbf{B}}

\newcommand{\bH}{\mathbf{H}}

\newcommand{\bW}{\mathbf{W}}\newcommand{\bX}{\mathbf{X}}
\newcommand{\bY}{\mathbf{Y}}\newcommand{\bZ}{\mathbf{Z}}

\newcommand{\fH}{\mathfrak{H}}

\newcommand{\dE}{\mathbb{E}}
\newcommand{\dH}{\mathbb{H}}

\newcommand{\dR}{\mathbb{R}}

\begin{document}

\begin{abstract}
  In this note, we derive a new logarithmic Sobolev inequality for the heat
  kernel on the Heisenberg group. The proof is inspired from the historical
  method of Leonard Gross with the Central Limit Theorem for a random walk.
  Here the non commutative nature of the increments produces a new gradient
  which naturally involves a Brownian bridge on the Heisenberg group. This new
  inequality contains the optimal logarithmic Sobolev inequality for the
  Gaussian distribution in two dimensions. We compare this new inequality with
  the sub-elliptic logarithmic Sobolev inequality of Hong-Quan Li and with the
  more recent inequality of Fabrice Baudoin and Nicola Garofalo obtained using
  a generalized curvature criterion. Finally, we extend this inequality to the
  case of homogeneous Carnot groups of rank two.

  \textsc{Résumé.}\selectlanguage{french} %
  Dans cette note, nous obtenons une inégalité de Sobolev logarithmique
  nouvelle pour le noyau de la chaleur sur le groupe de Heisenberg. La preuve
  est inspirée de la méthode historique de Leonard Gross à base de théorème
  limite central pour une marche aléatoire. Ici la nature non commutative des
  incréments produit un nouveau gradient qui fait intervenir naturellement un
  pont brownien sur le groupe de Heisenberg. Cette nouvelle inégalité contient
  l'inégalité de Sobolev logarithmique optimale pour la mesure gaussienne en
  deux dimensions. Nous comparons cette nouvelle inégalité avec l'inégalité
  sous-elliptique de Hong-Quan Li et avec les inégalités plus récentes de
  Fabrice Beaudoin et Nicola Garofalo obtenues avec un critère de courbure
  généralisé. Enfin nous étendons notre inégalités au cas des groupes de
  Carnot homogène de rang deux. %
  \selectlanguage{english}
 
\end{abstract}

\maketitle

{\footnotesize\tableofcontents}

\section{The Heisenberg group and our main result}
\label{se:intro}

In this note, we derive a new logarithmic Sobolev inequality for the heat
kernel on the Heisenberg group (Theorem \ref{th:}). Our proof is inspired from
the historical method of Leonard Gross based on a random walk and a Central
Limit Theorem. Due to the non commutative nature of the group structure, the
energy which appears in the right hand side involves an integral over some
Brownian bridges on the Heisenberg group. To compare with other logarithmic
Sobolev inequalities, we study Brownian bridges on the Heisenberg group and
deduce a weighted logarithmic Sobolev inequality (Corollary \ref{co:}). This
weighted inequality is close to the symmetrized version of the sub-elliptic
logarithmic Sobolev inequality of Hong-Quan Li. We also compare with
inequalities due to Fabrice Baudoin and Nicola Garofalo, and provide a short
semigroup proof of these inequalities in the case of the Heisenberg group.

We choose to focus on the one dimensional Heisenberg group, for simplicity;
and also because very precise estimates and results are known in this
particular case, which helps to compare our new inequality with existing ones.
Nevertheless our new logarithmic Sobolev inequality remains more generally
valid for homogeneous Carnot groups of rank two (Theorem \ref{th:carnot}).

\subsection*{The model} Let us briefly introduce the model and its main
properties. The Heisenberg group $\dH$ is a remarkable simple mathematical
object, with rich algebraic, geometric, probabilistic, and analytic aspects.
Available in many versions (discrete or continuous; periodic or not), our work
focuses on the continuous Heisenberg group $\dH$, formed by the set of
$3\times3$ matrices
\[
M(a,b,c)
=
\begin{pmatrix}
  1 & a & c \\
  0 & 1 & b \\
  0 & 0 & 1 \\
\end{pmatrix}, \quad a,b,c\in\dR.
\]

The Heisenberg group $\dH$ is a non commutative sub-group of the general
linear group, with group operations
$M(a,b,c)M(a',b',c')=M(a+a',b+b',c+c'+ab')$ and $M(a,b,c)^{-1}=(-a,-b,-c+ab)$.
The neutral element $M(0,0,0)$ is called the origin. The Heisenberg group
$\dH$ is a Lie group i.e.\ a manifold compatible with group structure.

\subsubsection*{The Heisenberg algebra is stratified}
The Lie algebra $\fH$ i.e.\ the tangent space at the origin of $\dH$ is the
sub-algebra of $\mathcal{M}_3(\dR)$ given by the $3\times 3$ matrices of the
form
\[
\begin{pmatrix}
  0 & x & z \\
  0 & 0 & y \\
  0 & 0 & 0 \\
\end{pmatrix}, \quad x,y,z\in\dR.
\]
The canonical basis of $\dH$
\[
X:=
\begin{pmatrix}
  0 & 1 & 0 \\
  0 & 0 & 0 \\
  0 & 0 & 0 \\
\end{pmatrix},
\quad
Y:=
\begin{pmatrix}
  0 & 0 & 0 \\
  0 & 0 & 1 \\
  0 & 0 & 0 \\
\end{pmatrix},
\quad\text{and}\quad
Z:=
\begin{pmatrix}
  0 & 0 & 1 \\
  0 & 0 & 0 \\
  0 & 0 & 0 \\
\end{pmatrix}.
\]
satisfies an abstract version of the Dirac (or annihilation-creation)
commutation relation
\[
[X,Y]:=XY-YX=Z\quad\text{and}\quad [X,Z]=[Y,Z]=0.
\]
This relation shows that the Lie algebra $\fH$ is stratified
\[
  \fH=\fH_{0} \oplus \fH_{1},
\]
where $\fH_{0} = \mathrm{span}(X,Y)$ and $\fH_{1} = \mathrm{span}(Z)$ is the
center of $\fH_{0}$. This makes the Baker-Campbell-Hausdorff formula on $\fH$
particularly simple:
\[
\exp(A)\exp(B)=\exp\left(A+B+\frac{1}{2}[A,B]\right),\quad A,B\in\fH.
\]
\subsubsection*{Exponential coordinates}
Lie groups such as $\dH$ with stratified Lie algebra (that is Carnot groups)
have a diffeomorphic exponential map $\exp:A\in\fH\mapsto\exp(A)\in\dH$. This
identification of $\dH$ with $\fH$, namely
\[
\begin{pmatrix}
  1 & a & c \\
  0 & 1 & b \\
  0 & 0 & 1 \\
\end{pmatrix}
\equiv
\exp
\begin{pmatrix}
  0 & x & z \\
  0 & 0 & y \\
  0 & 0 & 0 \\
\end{pmatrix}
=\exp(xX+yY+zZ),
\]
allows to identify $\dH$ with $\dR^3$ equipped with the group structure
\[
(x,y,z)\cdot(x',y',z')=(x+x',y+y',z+z'+\frac{1}{2}(xy'-yx'))
\]
and $(x,y,z)^{-1}=(-x,-y,-z)$. The identity element is the ``origin''
$e:=(0,0,0)$. From now on, we use these ``exponential coordinates''.
Geometrically, the quantity $\frac{1}{2}(xy'-yx')$ is the algebraic area in
$\dR^2$ between a piecewise linear path and its chord namely the area
between \[[(0,0),(x,y)]\cup[(x,y),(x+x',y+y')] \quad \text{and} \quad
[(0,0),(x+x',y+y')].\] This area is zero if $(x,y)$ and $(x',y')$ are
collinear. The group product
\[
(x,y,0)(x',y',0)=(x+x',y+y',\frac{1}{2}(xy'-yx'))
\]
in $\dH$ encodes the sum of increments in $\dR^2$ and computes automatically
the generated area.

\subsubsection*{Vector fields on $\dH$}
Elements of $\fH$ can classically be extended to left-invariant vector fields.
This identification will always be made implicitly and the same notation for
element of $\fH$ and vector field is used. This gives for the canonical basis
at a point $(x, y, z)$
\begin{equation}\label{eq:XYZ}
  X:=\partial_x-\frac{y}{2}\partial_z,\quad
  Y:=\partial_y+\frac{x}{2}\partial_z,\quad
  Z:=\partial_z.
\end{equation}

\subsubsection*{Metric structure of $\dH$} On the Heisenberg group, a natural
distance associated to the left-invariant diffusion operator $ L= \frac{1}{2}
(X^2+Y^2 + \beta^2 Z^2)$, $\beta \geq 0$, is defined for all $h,g\in\dH$ by
\[
d(h,g):=\sup_{f}(f(h)-f(g))
\]
where the supremum runs over all $f\in\mathcal{C}^\infty(\dH,\dR)$ such that
\[
\Gamma(f):=(Xf)^2+(Yf)^2+ \beta^2 (Zf)^2 \leq1.
\]
In the case $\beta >0$, this distance corresponds to the Riemannian distance
obtained by asserting that $(X,Y, \beta Z)$ is an orthonormal basis of the
tangent space in each point. In the case $\beta=0$, it is known, see for
instance \cite[Prop.~3.1]{MR922334}, that it coincides with the
Car\-not\,--\,Cara\-théo\-do\-ry sub-Riemannian distance obtained by taking
the length of the shortest horizontal curve. Recall that a curve is
\emph{horizontal} if its speed vector belongs almost everywhere to the
\emph{horizontal space} $\textrm{Vect} \{ X,Y\}$, and that the length of a
horizontal curve is computed asserting that $(X,Y)$ is an orthonormal basis of
this horizontal space in each point.

The Heisenberg group $\dH$ is topologically homeomorphic to $\dR^3$ and the
Lebesgue measure on $\dR^3$ is a Haar measure of $\dH$ (translation invariant)
but in the case $\beta=0$ the Hausdorff dimension of the $\dH$ for the
Carnot\,--\,Carathéodory metric is $4$.

Moreover, in the sub-elliptic case $\beta=0$, the Car\-not-Cara\-théo\-do\-ry
distance admits the following continuous family of dilation operators:
\[\Dil_\lambda(x,y,z)=(\lambda x,\lambda y,\lambda^2 z); \; \lambda>0.\]

A well known fact is that the Car\-not-Cara\-théo\-do\-ry distance is
equivalent to all homogeneous norm, see for instance \cite[Prop.\
5.1.4]{MR2363343}. In particular there exist constants $c_2>c_1>0$ such that
\begin{equation}\label{eq:d}
c_1( r^2+|z|)  \leq d(e,g)^2 \leq c_2( r^2+|z|);
\end{equation}
for all $ g= (x,y,z)\in \dH$ and $r^2:=x^2 +y^2$. 

\subsubsection*{Random walks on $\dH$}
Let $\beta\geq0$ be a real parameter. Let $(x_n,y_n,z_n)_{n\geq0}$ be
independent and identically distributed random variables on $\dR^3$ (not
necessarily Gaussian) with zero mean and covariance matrix
$\diag(1,1,\beta^2)$. Now set $S_0:=0$ and for all $n\geq1$,
\begin{equation}\label{eq:Sn}
  S_n:=(X_n,Y_n,Z_n):=
  \Bigr(\frac{x_1}{\sqrt{n}},\frac{y_1}{\sqrt{n}},\frac{z_1}{\sqrt{n}}\Bigr)
  \cdots
  \Bigr(\frac{x_n}{\sqrt{n}},\frac{y_n}{\sqrt{n}},\frac{z_n}{\sqrt{n}}\Bigr).
\end{equation}
The sequence ${(S_n)}_{n\geq0}$ is a random walk on $\dH$ started from the
origin and with i.i.d.\ ``non commutative multiplicative increments'' given by
a triangular array.
In exponential coordinates,
\[
X_n = \frac{1}{\sqrt{n}}\sum_{i=1}^n x_i,\quad
Y_n = \frac{1}{\sqrt{n}}\sum_{i=1}^n y_i,\quad
Z_n = A_n+\frac{1}{\sqrt{n}} \sum_{i=1}^{n} z_{i}
\]
where
\[
A_n := \frac{1}{2n} \sum_{i=1}^n \sum_{j=1}^{n} x_i \epsilon_{ij}  y_j
\quad\text{and}\quad
\epsilon_{i,j}:=\mathbf{1}_{j>i}-\mathbf{1}_{j<i}.
\]
The random variable $A_n$ is the algebraic area between the path
${(X_k,Y_k)}_{0\leq k\leq n}$ of a random walk in $\dR^2$ and its chord
$[(0,0),(X_n,Y_n)]$. With 
\[
\Dil_t(x,y,z)=(tx,ty,t^2z)
\]
being the dilation operator on $\dH$, we have
\[
(X_n,Y_n,A_n)=\Dil_{\frac{1}{\sqrt{n}}}((x_1,y_1,0)\cdots(x_n,y_n,0)).
\]
According to a Functional Central Limit Theorem (or Invariance Principle) on
Lie groups due to Daniel Stroock and Srinivasa Varadhan~\cite{MR0517406} (see
also Donald Wehn~\cite{MR0153042}, cited in \cite{MR1649603}),
\begin{align}\label{eq:FCLT}
  {\bigl(S_{\lfloor nt\rfloor}\bigr)}_{t\geq0}
  \quad 
  \underset{n\to\infty}{\overset{\text{law}}{\longrightarrow}} 
  \quad
  {\bigr(\bX_t,\bY_t,\bZ_t\bigr)}_{t\geq0}
  ={\bigr(\bX_t,\bY_t,\bA_t+ \beta\bW_t\bigr)}_{t\geq0}
\end{align}
where ${(\bX_t,\bY_t)}_{t\geq0}$ is a standard Brownian motion on $\dR^2$
started from the origin, where ${(\bW_t)}_{t\geq0}$ is a standard Brownian
motion on $\dR$ started from the origin and independent of
${(\bX_t,\bY_t)}_{t\geq0}$, and where $(\bA_t)_{t\geq0}$ is the Lévy area of
${(\bX_t,\bY_t)}_{t\geq0}$, in other words the algebraic area between the
Brownian path and its chord, seen as a stochastic integral:
\[
\bA_t:=\frac{1}{2}\Bigr(\int_0^t\!\bX_s\dd\bY_s-\int_0^t\!\bY_s\dd\bX_s\Bigr).
\]
\subsubsection*{The heat process on $\dH$}
The stochastic process $(\bH_t)_{t\geq0}=(h\cdot
(\bX_t,\bY_t,\bZ_t))_{t\geq0}$ started from $\bH_0=h$ is a Markov diffusion
process on $\dR^3$ admitting the Lebesgue measure as an invariant and
reversible measure. The Markov semigroup ${(P_t)}_{t\geq0}$ of this process is
defined for all $t\geq0$, $h\in\dH$, and bounded measurable $f:\dH\to\dR$, by
\[
P_t(f)(h) := \Esp(f(\bH_t)\mid \bH_0=h).
\]
For all $t>0$ and $h\in\dH$, the law of $\bH_t$ conditionally on $\bH_0=h$
admits a density and
\[
P_t(f)(h)=\int_{\dH}\!f(g)p_t(h,g)\dd g.
\]
Estimates on the heat kernel $p_t$ are available, see
\cite{MR1776501,MR2719560,MR973806}. For instance when $\beta=0$, there exist
constants $C_2>C_1>0$ such that for all $g=(x,y,z)\in\dH$ and $t>0$,
\begin{equation}\label{eq:pt}
  \frac{C_1}{\sqrt{t^4+t^3rd(e,g)}} %
  \exp\Bigr(-\frac{d^2(e,g)}{4t}\Bigr)
  \leq p_t(e,g) 
  \leq\frac{C_2}{\sqrt{t^4+t^3rd(e,g)}}
  \exp\Bigr(-\frac{d^2(e,g)}{4t}\Bigr)
\end{equation}
where $d$ is the Carnot\,--\,Carathéodory distance and where $r^2:=x^2+y^2$.

Let us define the family of probability measures (which depends on the  parameter $\beta$)
\[
\gamma_t:=\mathrm{Law}(\bH_t\mid \bH_0=0)=P_t(\cdot)(0).
\]
The infinitesimal generator is the linear second order operator
\[
L=\frac{1}{2}(X^2+Y^2+\beta^2 Z^2)
\]
where $X,Y,Z$ are as in~\eqref{eq:XYZ}.
The Schwartz space $\mathrm{Schwartz}(\dH,\dR)$ of rapidly decaying
$\mathcal{C}^\infty$ functions from $\dH\equiv\dR^3$ to $\dR$ is contained in
the domain of $L$ and is stable by $L$ and by $P_t$ for all $t\geq0$. By the
Dirac commutation relations $[X,Y]=Z=\partial_z$ and $[X,Z]=[Y,Z]=0$, the
operator $L$ is hypoelliptic, and by the Hörmander theorem $P_t$ admits a
$\cC^\infty$ kernel. The operator $L$ is elliptic if $\beta>0$ and not
elliptic if $\beta=0$ (singular diffusion matrix).

The operator $L$ acts as the two dimensional Laplacian on functions depending
only on $x,y$ and not on $z$. The one parameter family of operators obtained
from $L$ when $\beta$ runs through the interval $[0,1]$ interpolates between
the sub-elliptic or sub-Riemannian Laplacian $\frac{1}{2}(X^2+Y^2)$ (for
$\beta=0$) and the elliptic or Riemannian Laplacian $\frac{1}{2}(X^2+Y^2+Z^2)$
(for $\beta=1$). The sub-Riemannian and Riemannian Brownian motions
${(\bH_t)}_{t\geq0}$ have independent and stationary (non commutative)
increments and are Lévy processes associated to non commutative)
convolution semigroups ${(P_t)}_{t\geq0}$ on $\dH$. When $\beta=0$ the
probability measures $\gamma_t$ behaves very well with respect to dilation,
can be seen as a Gaussian measure on $\dH$, and a formula (oscillatory
integral) for the kernel of $P_t$ was computed by Paul Lévy using Fourier
analysis. See the books~\cite{MR1867362,MR2154760,MR1439509} and references
therein for more information and details on this subject.
  
\subsubsection*{Logarithmic Sobolev inequalities}
The entropy of $f:\dH\to[0,\infty)$ with respect to a probability measure
$\mu$ is defined by
\[
\Ent_\mu(f) 
:= 
\Esp_\mu(\Phi(f))-\Phi(\Esp_\mu(f))
\quad\text{with}\quad
\Esp_\mu(f):=\int\!f\dd\mu
\]
where $\Phi(u)=u\log(u)$. A logarithmic Sobolev inequality is of the form 
\[
\Ent_{\mu}(f^2) \leq \int T(f) \dd \mu
\] 
where $T$ is a ``good'' functional quadratic form. The most classical version
involves $T = \Gamma$ and contains many geometrical informations. The book
\cite{MR3155209} contains a general introduction to Sobolev type functional
inequalities for diffusion processes. However (see the discussion below), the
classical ``carré du champ'' does not capture the whole geometry of $\dH$.
Define a weighted ``carré du champ'' $T_{a} = \Gamma + a \Gamma^{Z}$, where
$a$ is a function and $\Gamma^{Z} f = (Zf)^{2} = (\partial_{z} f)^{2}$. Such a
gradient will naturally arise in the logarithmic Sobolev inequality we derive from the non commutativity.

\subsection*{Main results}

We start with the left-invariant diffusion operator $ L= \frac{1}{2} (X^2+Y^2
+ \beta^2 Z^2)$ for $\beta \geq 0$ on the Heisenberg group. In the case
$\beta>0$, the operator is elliptic and it is not hard to see that a usual
logarithmic Sobolev inequality holds for its heat kernel. Usual means here
that the energy in the right hand side is given by the ``carré du champ''
operator $\Gamma$ associated to $L$. Indeed, for $\beta>0$, $L$ can then be
thought of as the Laplace\,--\,Beltrami operator of a Riemannian manifold
whose Ricci curvature is actually constant and the Bakry\,--\,Émery theory
applies. The case $\beta=0$, is much more involved and have attracted a lot of
attention. Indeed, the operator $L$ is not anymore elliptic but is still
sub-elliptic. The Ricci curvature tends to $-\infty$ when $\beta $ goes to 0
and the Bakry\,--\,Émery theory fails. In this situation, the ``carré du
champ'' operator contains only the horizontal part of the gradient. The
question whether a logarithmic Sobolev inequality holds was answered
positively by Hong-Quan Li in~\cite{MR2240167} (see \eqref{eq:h:lsi:li}), see
also \cite{MR2462581,MR2558178} and \cite{MR2124868}.

In a different direction, even if the classical Bakry\,--\,Émery theory fails,
Fabrice Baudoin and Nicola Garofalo developed in~\cite{baudoin-garofalo} a
generalization of the curvature criterion which is well adapted to the
sub-Riemannian setting. One can then obtain some (weaker) logarithmic Sobolev
inequalities with an elliptic gradient in the energy (see
\eqref{eq:lsi-ell-nu}).

Our approach is different. We follows the method developed by Leonard Gross in
\cite{MR0420249} for the Gaussian and in \cite{MR1161977} for the path space
on \emph{elliptic} Lie groups. It is based on the tensorization property of
the logarithmic Sobolev inequality and on the Central Limit Theorem for a
random walk. It applies indifferently both for the sub-elliptic ($\beta=0$) or
the elliptic ($\beta>0$) Laplacian on the Heisenberg group. At least when
$\beta >0$, our main result Theorem \ref{th:} below is in a way an explicit
version of the abstract Theorem 4.1 in \cite{MR1161977}.

The interest in our result is double: we compute explicitly for the first time
the gradient which appears in the right hand side of Theorem 4.1 in
\cite{MR1161977} in the case of the Heisenberg group for all $\beta\geq0$, and
we show, by looking at the case $\beta=0,$ that the method of Gross gives a
non degenerate result for a sub-Riemannian model. This is surprising and
unexpected.

 
The next theorem, that is the main result of the paper and is proved in Section~\ref{se:proof:th:}, states the logarithmic Sobolev inequality for $\gamma = \gamma_{1}$.
From the scaling property of the heat kernel, we can easily deduce a logarithmic Sobolev inequality for $\gamma_{t}$ for every $t > 0$.
\begin{theorem}[Logarithmic Sobolev inequality]\label{th:}
  For all $\beta\geq0$ and $f\in\mathrm{Schwartz}(\dH,\dR)$,
  \begin{equation}\label{eq:beta:lsi}
    \Ent_{\gamma}(f^2)  
    \leq  2\int_0^1\!\Esp(g(\bH_1,\bH_t))\dd t
  \end{equation}
  where for $h=(x,y,z)$ and $h'=(x',y',z')$,
  \begin{align*}
  g(h,h')
   := &\, \left(\partial_xf(h)-\frac{y-2y'}{2}\partial_zf(h)\right)^2
  +\left(\partial_yf(h)+\frac{x-2x'}{2}\partial_zf(h)\right)^2
  +\beta^2\left(\partial_z f(h)\right)^2\\
   = &\, ((X + y' Z)f(h))^{2} + ((Y - x' Z)f(h))^{2} + \beta^{2} (Zf(h))^{2}.
  \end{align*}
\end{theorem}

The shape of the right hand side of~\eqref{eq:beta:lsi} comes from the fact
that the increments are not commutative: the sum in $S_n$ produces along
\eqref{eq:FCLT} the integral from $0$ to $1$.

The following corollary is obtained via Brownian Bridge and heat
kernel estimates.

\begin{corollary}[Weighted logarithmic Sobolev inequality]\label{co:} 
  If $\beta=0$ then there exist a constant $C>0$ such that for all
  $f\in\mathrm{Schwartz}(\dH,\dR)$,
  \begin{equation}\label{eq:beta:lsi:w}
    \Ent_{\gamma}(f^{2})  
    \leq  2\Esp_{\gamma}\Bigr((\partial_x f)^2+(\partial_y f)^2
    +  C(1+x^2+y^2+|z|)(\partial_zf)^2\Bigr). 
  \end{equation}
\end{corollary}

Corollary~\ref{co:} is proved in Section~\ref{se:proof:co:}.

%

\subsection*{Structure of the paper}

Section \ref{se:comp} provides a discussion and a comparison with other
inequalities such as the inequality of H.-Q. Li and the ``elliptic''
inequality of Baudoin and Garofalo. Section~\ref{se:proof:th:} is devoted to
the proof of Theorem~\ref{th:} which is based on the method of Gross using a
random walk and the CLT. Section~\ref{se:proof:co:} provides the proof of
Corollary~\ref{co:} by using an expansion of \eqref{eq:beta:lsi}, a
probabilistic (Bayes formula), analytic (bounds for the heat kernel on $\dH$),
and geometric (bounds for the Carnot\,--\,Carathéodory distance) arguments for
the control of the density of the Brownian bridge. For completeness, a short
proof of the ``elliptic'' inequality of Baudoin and Garofalo in the case of
the Heisenberg group (inequalities \eqref{eq:lsi-ell-nu}-\eqref{eq:lsi-ell-w})
is provided in Section \ref{se:another}. Finally, in Section \ref{se:carnot}
we give the extension of our main result (Theorem \ref{th:}) to the case of
homogeneous Carnot groups of rank two (Theorem \ref{th:carnot}).


\section{Discussion and comparison with other inequalities}%
\label{se:comp}

\subsubsection*{Novelty}

Taking $\beta=0$ in~\eqref{eq:beta:lsi} provides a new sub-elliptic
logarithmic Sobolev inequality for the sub-Riemannian Gaussian law $\gamma$,
namely, for all $f\in\mathrm{Schwartz}(\dH,\dR)$,
\begin{equation}\label{eq:hlsi}
  \Ent_{\gamma}(f^2)
  \leq
  2\int_0^1\!\Esp(g(\bH_1,\bH_t))\dd t
\end{equation}
where
\[
g(h,h')
:=
{\left(\partial_xf(h)-\frac{y-2y'}{2}\partial_zf(h)\right)^2}
+{\left(\partial_yf(h)+\frac{x-2x'}{2}\partial_zf(h)\right)}^2.
\]

\subsubsection*{Horizontal optimality} 

The logarithmic Sobolev inequality~\eqref{eq:hlsi}, implies the optimal
logarithmic Sobolev inequality for the standard Gaussian distribution
$\cN(0,I_2)$ on $\dR^2$ with the Euclidean gradient, namely, for all
$f\in\mathrm{Schwartz}(\dR^2,\dR)$,
\[
\Ent_{\cN(0,I_2)}(f^2)\leq 2\Esp_{\cN(0,I_2)}((\partial_xf)^2+(\partial_yf)^2).
\]
To see it, it suffices to express~\eqref{eq:hlsi} with a function $f$ that
does not depend on the third coordinate $z$. This shows in particular the
optimality (minimality) of the constant $2$ in front of the right hand side in
the inequality of Theorem~\ref{th:} and in~\eqref{eq:hlsi}.
  
\subsubsection*{Poincaré inequality} 

Recall that the variance of $f:\dH\to\dR$ with respect to $\mu$ is 
\[
\Var_\mu(f)
:=\int\!\Phi(f)\dd\mu-\Phi\Bigr(\int\!f\dd\mu\Bigr)
\quad\text{where this time}\quad
\Phi(u)=u^2.
\]
As usual, the logarithmic Sobolev inequality~\eqref{eq:hlsi} gives a Poincaré
inequality by linearization.
More precisely, replacing $f$ by $1+\veps f$ in
\eqref{eq:hlsi} gives, as $\veps\to0$,
\[
\Var_{\gamma}(f) \leq \int_0^1\!\Esp(g(\bH_1,\bH_t))\dd t.
\]

\subsubsection*{Comparison with H.-Q. Li inequality} 

For $\beta=0$, Hong-Quan Li has obtained in~\cite{MR2240167} (see also
\cite{MR2462581,MR2124868} for a Poincaré inequality) the following
logarithmic Sobolev inequality: there exists a constant $C_{\mathrm{LSI}}>0$
such that for all $f\in\mathrm{Schwartz}(\dH,\dR)$,
\begin{equation}\label{eq:h:lsi:li}
  \Ent_{\gamma}(f^2)\leq C_{\mathrm{LSI}}\Esp_{\gamma}((Xf)^2+(Yf)^2).
\end{equation}
The right hand side in~\eqref{eq:h:lsi:li} involves the ``carré du champ'' of
the sub-Laplacian $L$, namely the functional quadratic form: $\Gamma
(f,f):=\frac{1}{2} ( L(f^2)-2fLf) =(Xf)^2+(Yf)^2$. Following the by now
standard Bakry\,--\,Émery approach, the expansion of the scaled version shows
that necessarily $C_{\mathrm{LSI}} > 2$ but the optimal (minimal) constant
is unknown.

One can deduce from \eqref{eq:h:lsi:li} a weighted inequality. Namely, since
the random variables $-\bH_t$ and $\bH_t$ have the same law conditionally to
$\{\bH_0=0\}$, one can cancel out by symmetry, in average, the cross terms
involving $x\partial_xf\partial_z f$ and $y\partial_xf\partial_z f$ when
expanding the right hand side of the sum in~\eqref{eq:h:lsi:li} and its
rotated version. The symmetrized version of~\eqref{eq:h:lsi:li} that we
obtained in this way appears as a weighted logarithmic Sobolev inequality: for
all $f\in\mathrm{Schwartz}(\dH,\dR)$,
\begin{equation}\label{eq:h:lsi:li:sym}
  \Ent_{\gamma}(f^2)\leq 
  C_{\mathrm{LSI}}\Esp_{\gamma}
  \Bigr((\partial_xf)^2+(\partial_yf)^2+\frac{x^2+y^2}{4}(\partial _zf)^2\Bigr).
\end{equation}

\subsubsection*{Comparison with the  ``elliptic'' inequality of Baudoin-Garofalo}

Baudoin and Garofalo have developed in~\cite{baudoin-garofalo} a
generalization of the Bakry\,--\,Émery semigroup/curvature approach well
adapted to the sub-Riemannian setting, see also \cite[Prop.~4.11]{MR3601645},
\cite{MR2885961}, and \cite[Prop.~5.3.7~p.~129]{bonnefont-these}. Their
framework is well-suited for studying weighted functional inequalities such as
\eqref{eq:beta:lsi:w} and \eqref{eq:h:lsi:li:sym}. More precisely, it allows
first to derive the following result: if $\beta=0$ then for all real number
$\nu>0$ and all function $f\in\mathrm{Schwartz}(\dH,\dR)$,
\begin{equation}\label{eq:lsi-ell-nu}
  \Ent_{\gamma}(f^2)  
  \leq 2 \nu (\e^\frac{1}{\nu}-1) 
  \Esp_\gamma\Bigr((Xf)^2 + (Yf)^2 + \nu (Zf)^2\Bigr).
\end{equation}
The symmetrized version of~\eqref{eq:lsi-ell-nu} is given by the following new
weighted logarithmic Sobolev inequality:
if $\beta=0$ then for all real number $\nu>0$ and all function
$f\in\mathrm{Schwartz}(\dH,\dR)$,
\begin{equation}\label{eq:lsi-ell-w}
  \Ent_{\gamma}(f^2)  
  \leq 2 \nu (e^\frac{1}{\nu}-1)
  \Esp_\gamma\Bigr((\partial_x f)^2 + (\partial_y f)^2 
  + \bigr( \nu  + \frac{x^2+y^2}{4}\bigr)(\partial_z f)^2\Bigr).
\end{equation}
Our weighted inequality \eqref{eq:beta:lsi:w} is close to the weighted
inequalities \eqref{eq:h:lsi:li:sym} and \eqref{eq:lsi-ell-w}.

For the reader convenience, we provide a short proof of
\eqref{eq:lsi-ell-nu} and \eqref{eq:lsi-ell-w} in Section~\ref{se:another}.
This proof is the Heisenberg group specialization of the proof given in
\cite[Prop.~5.3.7~p.~129]{bonnefont-these} (PhD thesis of the first author)
see also \cite[Prop.~4.11]{MR3601645} (PhD advisor of the first author). 

\subsubsection*{Extensions and open questions.}
%

The Heisenberg group is the simplest non-trivial example of a Carnot group in
other words stratified nilpotent Lie group. Those groups have a strong
geometric meaning both in standard and stochastic analysis, see for
instance~\cite{MR2154760} for the latter point. Theorem~\ref{th:} is extended
to homogeneous Carnot group of rank two in Section \ref{se:carnot}. Note that
the criterion of Baudoin and Garofalo \cite{baudoin-garofalo} holds for Carnot
groups of rank two and an inequality similar to \eqref{eq:lsi-ell-nu} holds in
this context, see \cite{MR3601645}.

The bounds on the distance and the heat kernel used to derive the weighted
inequality~\eqref{eq:beta:lsi:w} are not available for general Carnot groups
and it should require more work to obtain an equivalent of
Corollary~\ref{co:}. As a comparison, note that a version
of~\eqref{eq:h:lsi:li} exists on groups with a so called H-structure,
see~\cite{MR2557945}, but a general version on Carnot groups is unknown due to
the lack of general estimates for the heat kernel. An extension of Theorem
\ref{th:} in the case of higher dimensional Carnot groups or in the case of
curved sub-Riemannian space as $CR$ spheres or anti-de Sitter spaces is
opened.

Moreover, in the context and spirit of the work of Leonard
Gross~\cite{MR1161977} in the elliptic case, an approach at the level of paths
space should be available. It is also natural to ask about a direct analytic
proof or semigroup proof of the inequality of Theorem~\ref{th:}, without using
the Central Limit Theorem.

\section{Proof of Theorem~\ref{th:}}
\label{se:proof:th:}

Fix a real $\beta\geq0$. Consider ${(x_n,y_n,z_n)}_{n\geq1}$ a sequence of
independent and identically distributed random variables with Gaussian law of
mean zero and covariance matrix $\mathrm{diag}(1,1,\beta^2)$. Let $S_n$ be as
in~\eqref{eq:Sn}. The Central Limit Theorem gives
\[
S_n\underset{n\to\infty}{\overset{\mathrm{law}}{\longrightarrow}} \gamma.
\]
The law $\nu_n$ of $S_n$ satisfies $\nu_n = (\mu_n)^{*n}$ where the
convolution takes place in $\dH$ and where $\mu_n$ is the Gaussian law on
$\dR^{3}$ with covariance matrix $\diag(1/n, 1/n, \beta^{2}/n)$.

For all $i=1,\ldots,n$, let us define
\[
S_{n,i} := (X_n, Y_n, Z_n, X_{n,i}, Y_{n,i})
\]
where
\[
X_{n,i} := - \frac{1}{\sqrt n} \sum_{j=1}^n \epsilon_{ij} x_j
\quad\text{and}\quad
Y_{n,i} :=  - \frac{1}{\sqrt n} \sum_{j=1}^n \epsilon_{ij} y_j.
\]
The optimal logarithmic Sobolev inequality for the standard Gaussian measure
$\cN(0,I_{3n})$ on $\dR^{3n}$ gives, for all
$g\in\mathrm{Schwartz}(\dR^{3n},\dR)$,
\[
\Ent_{\cN(0,I_{3n})}(g^2) \leq %
2\Esp_{\cN(0,I_{3n})}\Bigr( \sum_{i=1}^{n} (\partial_{x_{i}} g)^{2} %
+ (\partial_{y_{i}} g)^{2} + (\partial_{z_{i}} g)^{2} \Bigr).
\]
Let $s_n:\dR^{3n}\to\dH$ be the map such that
$S_n=s_n((x_1,y_1,z_1),\ldots,(x_n,y_n,z_n))$. For some
$f\in\mathrm{Schwartz}(\dH,\dR)$ the function $g = f(s_n)$ satisfies
\begin{align*}
  \partial_{x_i} g(s_n) 
  &= \frac{1}{\sqrt n}\Bigr(\partial_xf -\frac{Y_{n,i}}{2}\partial_zf\Bigr)(s_n),\\
  \partial_{y_i} g(s_n) 
  &= \frac{1}{\sqrt n}\Bigr(\partial_yf + \frac{X_{n,i}}{2} \partial_zf\Bigr)(s_n),\\
  \partial_{z_i} g(s_n) 
  &= \frac{\beta}{\sqrt n}(\partial_zf)(s_n).
\end{align*} 
It follows that for all $f\in\mathrm{Schwartz}(\dH,\dR)$, denoting
$\nu_n:=\mathrm{Law}(S_n)$,
\begin{equation}\label{eq:hlsi-step}
  \Ent_{\nu_n}(f^2) %
  \leq \frac{2}{n} \sum_{i=1}^{n}\Esp(h(S_{n,i})) %
  = 2\Esp\Bigr(\frac{1}{n}\sum_{i=1}^{n}h(S_{n,i})\Bigr)
\end{equation}
where $h:\dR^5\to\dR$ is defined from $f$ by
\[
h(x,y,z,x',y'):=
((\partial_x-\frac{y'}{2}\partial_z)f(x,y,z))^2
+
((\partial_y+\frac{x'}{2}\partial_z)f(x,y,z))^2
+ \beta^2  (\partial_z f(x,y,z))^2
\]
The right hand side of~\eqref{eq:hlsi-step} as $n\to\infty$ is handled by
explaining the law of $S_{n,i}$ through a triangular array of increments of
the process we anticipated in the limit. More precisely, let
${((\bX_t,\bY_t))}_{t\geq0}$ be a standard Brownian motion on $\dR^2$ started
from the origin, let ${(\bA_t)}_{t\geq0}$ be its Lévy area, and let
${(\bW_t)}_{t\geq0}$ be a Brownian motion on $\dR$ starting from the origin,
independent of ${(\bX_t,\bY_t)}_{t\geq0}$. Let us define, for $n\geq 1$ and
$1\le i \le n$,
\[
\xi_{n,i} := \sqrt n \left(\bX_{\frac{i}{n}} -  \bX_{\frac{i-1}{n}}\right),\quad
\eta_{n,i} := \sqrt n \left(\bY_{\frac{i}{n}} - \bY_{\frac{i-1}{n}}\right),\quad
\zeta_{n,i} := \sqrt n \left(\bW_{\frac{i}{n}} - \bW_{\frac{i-1}{n}}\right).
\]
For all fixed $n\geq1$, the random variables ${(\xi_{n,i})}_{1\leq i\leq n}$,
${(\eta_{n,i})}_{1\leq i\leq n}$, and ${(\zeta_{n,i})}_{1\leq i\leq n}$ are
independent and identically distributed with Gaussian law $\mathcal{N}(0,1)$.
Let us define now
\begin{align*}
  X_n & %
  :=  \frac{1}{\sqrt n} \sum_{i=1}^n \xi_{n,i},
  &
  Y_n & %
  := \frac{1}{\sqrt n} \sum_{i=1}^n \eta_{n,i},\\
  A_n & %
  :=\frac{1}{2n} \sum_{i=1}^n \sum_{j=1}^{n}\xi_{n,i}\epsilon_{i,j}  \eta_{n,i},
  &
  Z_n & %
  :=\beta\frac{1}{\sqrt n}\sum_{i=1}^n\zeta_{n,i} +A_n,\\
  X_{n,i} & %
  :=-\frac{1}{\sqrt n} \sum_{j=1}^n\epsilon_{i,j}\xi_{n,j},
  &
  Y_{n,i} & %
  :=-\frac{1}{\sqrt n} \sum_{j=1}^n\epsilon_{i,j}\eta_{n,i}.
\end{align*}
We have then the equality in distribution
\[
(X_n,Y_n,Z_n-A_n,X_{n,i},Y_{n,i}) %
\overset{d}{=} %
(\bX_1,\bY_1,\beta \bW_1, %
\bX_1-(\bX_{\frac{i-1}{n}}+\bX_{\frac{i}{n}}), %
\bY_1-(\bY_{\frac{i-1}{n}}+\bY_{\frac{i}{n}})).
\]
Moreover, as $i/n\to s\in[0,1]$, we have the convergence in distribution
\[
S_{n,i}:=(X_n,Y_n,Z_n,X_{n,i},Y_{n,i}) %
\underset{n\to\infty}{\overset{d}{\longrightarrow}} %
(\bX_1,\bY_1,\bA_1+\beta\bW_1,\bX_1-2\bX_s,\bY_1- 2\bY_s).
\]
It follows that for all continuous and bounded $h:\dR^5\to\dR$,
\begin{align*}
  \frac{1}{n}\sum_{i=1}^n\Esp(h(S_{n,i}))
  &=\int_0^1\!\Esp(
  h(X_n,Y_n,Z_n,X_{n,\frac{\lfloor tn\rfloor}{n}},Y_{n,\frac{\lfloor tn\rfloor}{n}})\dd t\\
  &\underset{n\to\infty}{\longrightarrow}
  \int_0^1\!\Esp(h(\bX_1,\bY_1,\bA_1+\beta\bW_1,\bX_1-2\bX_t,\bY_1-2\bY_t))\dd t.
\end{align*}

\section{Proof of Corollary~\ref{co:}}
\label{se:proof:co:}

Let us consider~\eqref{eq:beta:lsi} with $\beta=0$. By expanding the
right-hand side, and using the fact that the conditional law of $\bH_1$ given
$\{\bH_0=0\}$ is invariant by central symmetry, we get a symmetrized weighted
logarithmic Sobolev inequality: for all $f\in\mathrm{Schwartz}(\dH,\dR)$,
\begin{align*}
  \Ent_{\gamma}(f^2)
  &\leq 2\Esp\bigr((\partial_x f)^2(\bH_1) +( \partial_y f)^2(\bH_1) \bigr)\\
  &\quad+\frac{1}{2}\int_0^1\!\Esp\bigr(\Esp\bigr((\bX_1- 2\bX_t)^2
  +(\bY_1- 2\bY_t )^2\mid \bH_1\bigr)( \partial_z f)^2(\bH_1)\bigr)\dd t\\
  &\leq  2\Esp\bigr(( \partial_x f)^2(\bH_1)+( \partial_y f)^2(\bH_1)\bigr)\\
  &\quad+ \Esp\bigr((\bX_1^2+\bY_1^2)(\partial_zf)^2(\bH_1)\bigr)
  +4\Esp\Bigr((\partial_z f)^2(\bH_1)\int_0^1\!\Esp\bigr(\bX_t^2+\bY_t^2\mid
  \bH_1\bigr)\dd t\Bigr).
\end{align*}

The desired result is a direct consequence of the following lemma.

\begin{lemma}[Bridge control]\label{le:bridge}
  There exists a constant $C>0$ such that for all $0\leq t \leq 1$ and
  $h=(x,y,z)\in\dH$,
  \[
  \Esp(\bX_t^2+\bY_t^2 \mid\bH_1=h) \leq C ( t^2 d^2(e,h) +t)
  \]
  and 
  \[
  \int_0^1\!\Esp(\bX_t^2+\bY_t^2\mid \bH_1=h)\dd t \leq C(1+x^2+y^2+|z|).
  \]
\end{lemma}

Note that for the classical euclidean Brownian motion 
\[
\Esp(\Vert\bB_t\Vert^2\mid\bB_1)=t^2\Vert\bB_1\Vert^2+nt(1-t).
\]

\begin{proof}[Proof of Lemma~\ref{le:bridge}]
  For all random variables $U,V$, we denote by $\varphi_U$ the density of $U$
  and by $\varphi_{U\mid V=v}$ the conditional density of $U$ given $\{V=v\}$.
  For all $0<t\leq1$ and $k\in\dH$,
  \[
  \Esp(\bX_t^2+\bY_t^2\mid\bH_1=k)
  =\int_{\dH}\!r^2_g\,\varphi_{\bH_t\mid\bH_1=k}(g)\dd g
  \]
  where $r_g^2=x_g^2+y_g^2$ and $g=(x_g,y_g,z_g)$. Recall that
  $p_t(h,g):=\varphi_{\bH_t=g\mid\bH_0=h}$ and that $e:=(0,0,0)$ is the origin
  in $\dH$. Thanks to the Bayes formula, for all $g,k\in\dH$,
  \[
  \varphi_{\bH_t\mid\bH_1=k}(g)
  = \frac{\varphi_{(\bH_t,\bH_1)}(g,k)}{\varphi_{\bH_1}(k)}%
  = \frac{\varphi_{\bH_t}(g)\varphi_{\bH_1\mid\bH_t=g}(k)}{\varphi_{\bH_1}(k)}%
  = \frac{p_t (e,g)p_{1-t} (g,k)} {p_1(e,k)}.
  \] 
  Back to our objective, we have
  \begin{align*}
    \Esp(\bX_t^2+\bY_t^2\mid\bH_1=k)
    & =\frac{\displaystyle\int\!p_t(e,g)\,r^2_g\,p_{1-t}(g,k)\dd g}{p_1(e,k)}\\
    &\leq\frac{\displaystyle\int\!p_t(e,g)d^2(e,g) p_{1-t}(g,k)\dd g}
    {p_1(e,k)}.
  \end{align*}  
  In what follows, the constant $C$ may change from line to line. The idea is
  to kill the polynomial term $d^2$ in the numerator by using the exponential
  decay of the heat kernel, at the price of a slight time change. Namely,
  using~\eqref{eq:pt}, we get, for all $0<\veps<1$,
  \begin{align*}
    \int\!p_t(e,g)r^2(e,g)p_{1-t}(g,k)\dd g 
    &\leq\int\!\frac{C}{\sqrt{t^4+t^3rd(e,g)}}d^2(e,g)
    \exp\left(-\frac{d^2(e,g)}{4t}\right)p_{1-t}(g,k)\dd g\\
    &\leq\int\!\frac{Ct}{\veps\sqrt{t^4+t^3rd(e,g)}}
    \exp\left(-(1-\veps)\frac{d^2(e,g)}{4t}\right)p_{1-t}(g,k)\dd g\\
    &\leq\int\!\frac{Ct}{\veps}
    \sqrt\frac{{\left(\frac{t}{1-\veps}\right)^4
        +\left(\frac{t}{1-\veps}\right)^3rd(e,g)}}{{t^4+t^3rd(e,g)}}p_{\frac{t}{1-\veps}}(e,g)p_{1-t}(g,k)\dd g\\
    &\leq\frac{Ct}{\veps(1-\veps)^2}\int\!p_{\frac{t}{1-\veps}}(e,g)p_{1-t}(g,k)\dd g\\
    &\leq\frac{Ct}{\veps(1-\veps)^2}p_{1+\frac{\veps t}{1-\veps}}(e,k),
  \end{align*}
  where we used $x \exp(-x) \leq \frac{1}{e\veps} \exp (-(1-\veps )x )$.
  Therefore, for all $0<\veps\leq 1/2$,
  \begin{align*}
    \Esp(\bX_t^2+\bY_t^2\mid\bH_1 = k)
    &\leq\frac{Ct}{\veps (1-\veps)^2}  \frac{p_{1 + \frac{\veps t}{1-\veps}} (e,k)} {p_{1} (e,k)}\\
    &\leq\frac{Ct}{\veps (1-\veps)^2}  \exp \left( \frac{\veps t} {1 -\veps+\veps t} \frac{d^2(e,k)}{4} \right) \\
    &\leq\frac{C t}{\veps}  \exp \left(\veps t \frac{d^2(e,k)}{2} \right).
  \end{align*} 
  Now if $t d^2(e,k) \geq 1$, then we take  $\veps =1/(2td^2(e,k))$ which gives 
  \[
  \Esp(\bX_t^2+\bY_t^2\mid\bH_1 = k)\leq C t^2 d^2(e,k),
  \]
  while if $t d^2(e,k) < 1$, then we take $\veps=1/2$ which gives
  \[
  \Esp(\bX_t^2+\bY_t^2\mid \bH_1 = k)\leq C t.
  \]
  This provides the first desired inequality. We get the second
  using~\eqref{eq:d}.
\end{proof}

\section{Proof of inequalities \eqref{eq:lsi-ell-nu} and \eqref{eq:lsi-ell-w}}
\label{se:another}

In this section, we provide for the reader convenience a short proof of the
inequalities \eqref{eq:lsi-ell-nu} and \eqref{eq:lsi-ell-w}. This corresponds
to the case $\beta=0$, but the method remains in fact valid beyond the
assumption $\beta=0$. Recall that this proof is essentially the Heisenberg
group specialization of the proof given in
\cite[Prop.~5.3.7~p.~129]{bonnefont-these} (see also
\cite[Prop.~4.11]{MR3601645}).

\begin{proof}[Proof of \eqref{eq:lsi-ell-nu}]
  For simplicity, we change $L$ by a factor 2 and set $L=X^2+Y^2$ and
  $P_t=\e^{tL}$, and in particular $\gamma=P_{1/2}(\cdot)(0)$. For all
  $f,g\in\mathrm{Schwartz}(\dH,\dR)$, let us define
  \begin{align*}
    \Gamma^{\mathrm{hori}}(f,g)&:=\frac{1}{2}(L (fg) -fLg -g Lf) = X(f) X(g) +Y(f) Y(g),\\
    \Gamma^{\mathrm{vert}}(f,g)&:=Z(f)Z(g),\\
    \Gamma^{\mathrm{elli}}(f,g)&:=  \Gamma^{\mathrm{hori}}(f,g) + \nu \Gamma^{\mathrm{vert}}(f,g).
  \end{align*}
  Let us also denote
  \begin{align*}
    \Gamma_2^{\mathrm{hori}}(f,f)
    &:=\frac{1}{2}(L \Gamma^{\mathrm{hori}}(f,f) -2 \Gamma^{\mathrm{hori}}(f, Lf)),\\ 
    \Gamma_2^{\mathrm{vert}}(f,f)
    &:=  \frac{1}{2}(L \Gamma^{\mathrm{vert}}(f,f) -2 \Gamma^{\mathrm{vert}}(f, Lf)),\\
    \Gamma_2^{\mathrm{mix}}(f,f)
    &:=  \frac{1}{2}(L \Gamma^{\mathrm{elli}}(f,f) -2 \Gamma^{\mathrm{elli}}(f, Lf)).
  \end{align*}
 In the sequel, we also denote $\Gamma(f)= \Gamma(f,f)$ and  $\Gamma_2(f)=\Gamma_2(f,f)$.
 
 \emph{Curvature inequality.} The following inequality holds: for all
 $f\in\mathrm{Schwartz}(\dH,\dR)$,
  \begin{equation}\label{eq:CD}
    \Gamma_2^{\mathrm{mix}}(f,f) \geq - \frac{1}{\nu}  \Gamma^{\mathrm{elli}} (f,f).
  \end{equation}
  Indeed, an easy computation gives  \begin{align*}
    \Gamma_2^{\mathrm{hori}}(f,f) 
    &= (X^2f)^2 + (Y^2f)^2 + (XY f)^2 + (YX f)^2 - 2 (Xf) (YZf) + 2 (Yf) (XZf)\\
  \Gamma_2^{\mathrm{vert}}(f,f) 
    &= (XZf)^2 + (YZf)^2.
  \end{align*}
  Since $ \Gamma_2^{\mathrm{mix}}=\Gamma_2^{\mathrm{hori}}+ \nu
  \Gamma_2^{\mathrm{vert}}$ and $Zf=XYf-YXf$, the Cauchy\,--\,Schwarz's inequality
  gives
  \[
  \Gamma_2^{\mathrm{mix}}(f,f) %
  \geq \frac{1}{2} (Lf)^2 %
  + \frac{1}{2} (XYf+YXf)^2 %
  +\frac{1}{2}  (Zf)^2  %
  -\frac{1}{\nu}X(f)^2 %
  -\frac{1}{\nu}Y(f)^2, 
  \]
 which implies the desired curvature inequality~\eqref{eq:CD}.
 
 \emph{Semigroup inequality.} Let $f\in\mathrm{Schwartz}(\dH,\dR)$ with
 $f\geq0$. For all $0\leq s\leq t$, set
 \[
   U(s):= (P_{t-s} f ) \, \Gamma^{\mathrm{hori}} \log(P_{t-s} f )
 \quad\text{and}\quad 
 V(s):= (P_{t-s} f ) \, \Gamma^{\mathrm{elli}} (\log(P_{t-s} f )).
 \]
 Then for $0\leq s\leq t$,  
 \begin{align}
   \label{eq:comparison_U}
   L U + \partial_s U &= (P_{t-s} f ) \,
   \Gamma^{\mathrm{hori}}(\log(P_{t-s}f)) %
   \leq (P_{t-s}f)\,\Gamma^{\mathrm{elli}}(\log(P_{t-s} f ))=V(s)\\
   \label{eq:comparison_V}
   L V + \partial_s V &=2 (P_{t-s} f ) \, \Gamma_2^{\mathrm{mix}}(
   \log(P_{t-s} f )) %
   \geq -\frac{2}{\nu} V(s).
 \end{align}
 The first equality in~\eqref{eq:comparison_U} holds since
 $\Gamma^{\mathrm{hori}}$ is the ``carré du champ'' associated to $L$, the
 inequality holds because $\Gamma^{\mathrm{vert}}(f, f) \geq 0$ and the second
 equality is the definition of $V$.
 \\
 In~\eqref{eq:comparison_V} we have used that $\Gamma^{\mathrm{elli}} =
 \Gamma^{\mathrm{hori}} + \nu \Gamma^{\mathrm{vert}}$, that the horizontal
 part of $\Gamma^{\mathrm{elli}}$ will produce $\Gamma_{2}^{\mathrm{hori}}$
 (same kind of computation as in the first inequality
 of~\eqref{eq:comparison_U}) and that vertical part commute to the horizontal
 one $\Gamma^{\mathrm{hori}} (f,\Gamma^{\mathrm{vert}} (f,f))=
 \Gamma^{\mathrm{vert}} (f, \Gamma^{\mathrm{hori}} (f,f))$, the inequality
 comes from~\eqref{eq:CD}.

 \emph{Final step.} Since by~\eqref{eq:comparison_V}, $L(\e^\frac{2s}{\nu} V(
 s)) + \partial_s(\e^\frac{2s}{\nu} V(s)) \geq 0$, a parabolic comparison such
 as~\cite[Prop.~4.5]{baudoin-garofalo} or a simple semigroup interpolation
 implies that for $t\geq 0$,
 \[
 \e^\frac{2t}{\nu} P_t( f \Gamma^{\mathrm{elli}} (\log f))   
 =\e^\frac{2t}{\nu} V(t) \geq V(0)= (P_t f) \Gamma^{\mathrm{elli}}( \log P_t f).
 \]
 In particular, 
 \begin{equation}
   \label{eq:elliptic_sub_commutation}
   V(s) \leq \e^\frac{2(t-s)}{\nu}  P_{t-s}( f \Gamma^{\mathrm{elli}} (\log f)).
 \end{equation}
 Now from~\eqref{eq:comparison_U} another application of the parabolic
 comparison theorem and the last estimate~\eqref{eq:elliptic_sub_commutation}
 give
 \[
 P_t (U(t)) 
 \leq U(0) 
 + \int_0^t P_s (V(s))\dd s \leq U(0) 
 +   \int_0^t \e^\frac{2(t-s)}{\nu}\dd s \; P_t \left( f \Gamma^{\mathrm{elli}} (\log f)\right);
 \]
 that is:
 \[
 P_t (f\log f)(x) - P_t (f)(x) \log P_t (f)(x) \leq \frac{\nu}{2} %
 \left(\e^{\frac{2t}{\nu}}-1\right)\, P_t \left( \frac{\Gamma^{\mathrm{elli}}
     ( f,f)} {f}\right)(x).
 \]
 The conclusion follows by taking $t=1/2$ and $x=0$ since $\gamma=P_{1/2}(\cdot)(0)$.
\end{proof}

\begin{proof}[Proof of \eqref{eq:lsi-ell-w}]
  Let us consider the right (instead of left) invariant vector fields
  \[
  \hat X:=\partial_x +\frac{y}{2} \partial_z
  \quad\text{and}\quad 
  \hat Y:=\partial_y -\frac{x}{2} \partial_z
  \]
  and $\hat L=\hat X+\hat Y$ and $\hat P_t=\e^{t\hat L}$ the corresponding
  generator and semi-group. The semi-group is bi-invariant in the sense that
  $P_t f(0)=\hat P_t f(0)$, see for instance~\cite{MR2462581}. Recall that
  $\gamma=P_{1/2}(\cdot)(0)=\hat P_{1/2}(\cdot)(0)$. The method of proof of
  \eqref{eq:lsi-ell-nu} remains valid if one replaces $X,Y,L,P_t$ by their
  right invariant counter parts and yields that for all
  $f\in\mathrm{Schwartz}(\dH,\dR)$,
  \begin{equation}\label{eq:lsi-ell2}
    \Ent_{\gamma}(f^2)  
    \leq  2 \nu (e^\frac{1}{\nu}-1) \Esp_\gamma\Bigr((\hat X f)^2 + (\hat Yf)^2 + \nu (- Zf)^2\Bigr).
  \end{equation}
  The conclusion follows by the summation of the
  inequalities~\eqref{eq:lsi-ell-nu} and~\eqref{eq:lsi-ell2}.
\end{proof}

\section{Extension to homogeneous Carnot groups of rank two}
\label{se:carnot}

In this final section we consider the class of homogeneous Carnot groups of
step two. We refer to \cite{MR2363343} for more details and results on this
class of Carnot groups. An homogeneous Carnot groups of step two is
$\dR^N=\dR^d \times \dR^m$ equipped with the group law given by
\[
(x,z) \cdot (x',z')=(x+x', z+z' + \frac{1}{2} \langle Bx,x' \rangle)
\]
where $x,x' \in \dR^d$, $z,z'\in \dR^m$ and 
\[
\langle Bx,x' \rangle = \left(\langle B^{(1)} x,x' \rangle, \cdots, \langle B^{(m)} x,x' \rangle \right) 
\]
for some linearly independent skew-symmetric $d\times d$ matrices $B^{(l)}$,
$1\leq l \leq m$. This class of groups includes a lot of usual examples, for
instance all the Heisenberg groups $\mathbb H_n$ and free rank two Carnot
groups. The case of the Heisenberg group $\mathbb H_1$ corresponds to
\[
  d=2,\ m=1,\ B=\begin{pmatrix} 0&-1\\1&0 \end{pmatrix}.
\]
Actually, it is known that each stratified group of rank two is isomorphic to
such an homogeneous Carnot group, see for instance \cite[Theorem
3.2.2]{MR2363343}. These homogeneous Carnot groups admit a dilation given by
\[
  \Dil_\lambda(x,z):= (\lambda x, \lambda^2 z).
\]
The natural sub-Riemannian Brownian motion is given by
$(\bX_t,\bZ_t)_{t\geq 0}$ where $\bX$ is a standard Brownian motion on $\dR^d$
and where $\bZ$ corresponds to its generalized Levy area:
\[
  \bZ_t^{(l)}= \sum_{1\leq p <q \leq d} b_{p,q}^{(l)} \bA^{(p,q)}_t
\]
with
\[
  \bA^{(p,q)}_t= \int_0^t \bX_s^{(p)} d\bX_s^{(q)} -  \int_0^t \bX_s^{(q)} d\bX_s^{(p)} .
\]
We denote by $\gamma$ the law of $(\bX_1,\bZ_1)$. The proof given in the case
of the Heisenberg group $\mathbb H_1$ easily extends to this setting and leads
to the following result.

\begin{theorem}[Logarithmic Sobolev inequality]\label{th:carnot}
  For all $f\in\mathrm{Schwartz}(\dR^N,\dR)$,
  \[
    \Ent_{\gamma}(f^2) \leq  2 \sum_{ p= 1}^d  \int_0^1   \dE\left[ \left(  \partial_p f (\bX_1,\bZ_1) + \sum_{l=1}^m 
        \left( \sum_{q=1}^ d b_{p,q}^{(l)}  (\bX_1^{(q)} -2 \bX_s^{(q)}) \right)      \partial_{d+l} f (\bX_1,\bZ_1) \right) ^2\right].
  \] 
\end{theorem}

\section*{Acknowledgements}

D.C.\ would like to thank Leonard Gross for his encouragement to explore this
problem and Fabrice Baudoin for his hospitality during a visit to Purdue
University in Fall 2008.

%
%
% ...


\makeatletter
\def\@MRExtract#1 #2!{#1}     
\renewcommand{\MR}[1]{
  \xdef\@MRSTRIP{\@MRExtract#1 !}%
  \href{http://www.ams.org/mathscinet-getitem?mr=\@MRSTRIP}{MR-\@MRSTRIP}}
\makeatother

\providecommand{\bysame}{\leavevmode\hbox to3em{\hrulefill}\thinspace}

\providecommand{\href}[2]{#2}

\end{document}